\documentclass[12pt,twoside,a4paper]{article}
\usepackage[cp1251]{inputenc}
\usepackage{srcltx}
\usepackage{amsmath,amsfonts,amscd,amssymb,latexsym}

\usepackage{amsmath,amsthm,amsfonts}
\usepackage{graphicx}
\usepackage{hyperref}
\usepackage[usenames]{color}

\newtheorem{theorem}{Theorem}
\newtheorem{lemma}{Lemma}
\theoremstyle{definition}

\newtheorem{definition}{Definition}

\newtheorem{problem}{Problem}
\newtheorem{remark}{Remark}
\newtheorem{proposition}{Proposition}

\begin{document}

{\large

\begin{center}
{\bf On verbally closed subgroups of free solvable groups}
\end{center}
\begin{center}
{\bf V. A. ROMAN'KOV\footnote{This research is supported  by program of fundamental investigations   of  SB RAS: I. 1.1.4, project 0314-2019-0004}, E. I. TIMOSHENKO\footnote{This research is supported  by RFBR, project18-01-00100a}}

\end{center}

Abstract: We study verbally closed subgroups of free solvable groups. A number of results is proved that give  sufficient conditions under whose a verbally closed subgroup is turned to be a retract and so algebraically closed of the full group.

\begin{center}
{\bf INTRODUCTION}
\end{center}
\medskip
 Algebraically closed objects play an important role in the modern algebra. 
In this paper we study   verbally and $l$-verbally closed  subgroups of free solvable groups. We establish a connection of them with  retracts  and algebraically  closed subgroups.  

Let ${\cal K}$ be a class of algebraic structures  of language  $\mathbf{L}$. A structure     $A \in {\cal K}$
is said to be {\it algebraically closed}  in  ${\cal K}$ if and only if for any  positive $\exists $-sentence   $\varphi (x_1, ..., x_n)$
in $\mathbf{L}$ with constants in  $A,$ if $\varphi (x_1, ..., x_n)$ is true in a structure  
$B \in {\cal K},$  that has  $A$ as a substructure, then  $\varphi (x_1, ..., x_n)$ is true in  $A.$
 See for instances, \cite{Hod}, \cite{Scott}.

An interesting specific case  arises when we take as
${\cal K}$ the class  sub$(B)$  of all substructures of a structure  $B.$ Moreover, we can take as  ${\cal K}$  a pair $\{A, B\},$  where $A$   is a substructure of $B.$ Then we can say about {\it algebraic closeness} of $A$ in $B.$

Now let $\mathbf{L}$ be the group-theoretical language and
${\cal K}$ be a class of groups. In this case, the notion above can be presented in pure group-theoretical  terminology.  For groups 
$H$ and $G$ we write $H \leq G,$ if  $H$ is a subgroup of $G,$  and we call  $G$ {\it extension}   of $H.$  By  an  {\it equation}  $w(x_1, ..., x_n; H)= h, h \in H,$ in {\it variables} ({\it unknowns}) $x_1, ..., x_n$ with constants  in $H$ we mean  an expression in which  $w(x_1, ..., x_n; H)$  is a group word in $x_1, ..., x_n$ and constants in $H.$  We say that an equation $w(x_1, ..., x_n;H) = h$  has a  {\it solution (is solvable)} in $H$ (or more generally, in an extension $G \geq H$)   if there is a tuple  of elements   $g_1, ..., g_n$  of  $H$ (or more generally of $G$)  if after a substitution them instead of  $x_1, ..., x_n,$  respectively, we have  the true equality $w(g_1, ..., g_n;H) = h. $ An equation of the form $w(x_1, ..., x_n) = h$,     in which the left-hand side does not include constants, but only variables, is called,    {\it split} equation.

In a natural way we define  a {\it system of equation} and a  {\it solution }  of it.

A notion of an algebraically closed subgroup of a group is rewritten as follows.
\begin{definition}
\label{de:1}
Let  $G$  be a group.
A subgroup   $H\leq G$ is called  {\it algebraically closed} in $G,$ if any  finite system of equations with constants in   $H$ that has a solution in     $G$ has a solution  in $H.$
\end{definition}
In   \cite{MR},    the following definition was given.
\begin{definition}
Let $G$  be a group. A subgroup  $H \leq G$ is called  {\it verbally closed} in      $G,$ if for any group word     $w(x_1, ..., x_n)$  in  variables   $ x_1, ..., x_n$  without constants and each  $h \in H$ a split equation 
\begin{equation}
\label{eq:1}
w(x_1, ..., x_n) = h,
\end{equation}
\noindent
that has a solution in  $G,$ is solvable in  $H.$
\end{definition}
This notion, more weak than algebraic closeness,  meanwhile corresponds to  much more strong notion of retract.
\begin{definition}
 {\it Retract}   of a group    $G$  is a subgroup  $H \leq G$  such that there is an endomorphism   $\rho :    G  \rightarrow   H$   that is identical on
$H.$  Such an endomorphism   $\rho $ is called  {\it retraction}.
\end{definition}
Let   $F_r$   be a free group of rank    $r.$  Obviously, each free factor of $F_r$ is retract. For    $r \geq 2$, $F_r$  admits retracts that are not free factors.  For example, any element of the form  $g = f_1u$ of    $F_r$ with basis $\{f_1, ..., f_r\}$ where  $u$  belongs to the normal closure of       $f_2, ..., f_r,$ generates a cyclic retract  $R$.  In this case, a retraction     $\rho : F_r \rightarrow R,$ is defined by map 
$ f_1 \mapsto g, f_j \mapsto 1$  for  $j = 2, ..., r,$ but a cyclic group generated by $g,$ is a free factor of $F_r$ if and only if   $g$ is primitive, i.e., it can be included in some basis of $F_r.$ If $r=2$ and $u$ lies in the commutant of  $F_2$, then  $g=f_1u$ is primitive if and  only if it conjugates to  $f_1$. This assertion is followed from  a classical Nielsen's result,  that every automorphism of  $F_2$ identical modulo commutant is inner.  Element $g=f_1f_2^{-1}f_1^{-1}f_2f_1$ of $F_2$ generates retract but not a free factor. 

\begin{remark}
\label{re:1}
If $H$ is a retract of  $G$ then  $H$  is algebraically closed in  $G$. Indeed, to obtain a solution in $H$ of an equation $w(x_1, ..., x_n; H) = h$  from a solution $x_i=g_i\,  (i=1, ..., n)$ in $G$, it is sufficient to apply a retraction $\rho : G \rightarrow H.$ Then $w(g_1, ..., g_n, H)=h$ implies $\rho (w(g_1, ..., g_n); \rho (H))=w(\rho (g_1), ..., \rho (g_n); H)=\rho (h)=h.$  Obviously that any algebraically closed subgroup   $H$ of $G$  is verbally closed in $G$. At the same time, among these three concepts there are no two  equivalent one to other. In thesis \cite{Mthesis}, there are examples of algebraically closed subgroups that are not retracts, as well as examples of verbally closed subgroups that are not algebraically closed.  
\end{remark}

In  \cite{Berg},   the author proves that intersection of a finite set of retracts of a free group  $F_r$   is retract.   In \cite{KT},  there is a question   17.19  (the same as the question   F11  in  \cite{Open}): If    $H$  is a finitely generated subgroup of  $F_r$ and  $R$  is a retract of $F_r,$  is it followed that the intersection $R \cap H$  is retract of  $H$? Recently this question was answered negatively in  \cite{Retract}.

Thus the concept of verbal closeness is  weaker than algebraic closeness. However in  \cite{MR} the authors proved that any verbally closed subgroup $H$  of a free group   $F_r$  ($r\in \mathbb{N}$) is a retract of $F_r$. It follows that the sets of verbally closed subgroups, algebraically closed subgroups and retracts of $F_r$ coincide.   

In \cite{RomKhis}, a similar statement  was shown for free nilpotent groups.  Let $N_{r,c}$ denote a free nilpotent group of a finite rank $r$ and nilpotency class $c.$ Then the following properties of a subgroup $H\leq N_{r,c}$  are equivalent one to other:  verbal closeness, algebraic closeness, the property ''to be a retract'' and the property ''to be a free factor'' with respect to the variety  $\mathfrak{N}_c$ of all nilpotent groups of nilpotency class $\leq c.$

In the series of papers  \cite{Maz1}, \cite{Maz2}, \cite{Mthesis}, \cite{KlMaz} and \cite{KMM} verbal closeness was studied. The following property was introduced. A group  $H$ is a  {\it strongly verbally closed}  if $H$ is an algebraically closed subgroup in any group containing $H$ as a verbally closed subgroup. The number of results was presented  showing that many groups are strongly verbally closed. 

 Now we introduce the following new notion.
 \begin{definition}
 Let  $G$ be a group.  For any $l\in \mathbb{N},$ a subgroup
$H\leq G$  is said to be  $l$-{\it verbally closed} in $G$ if every system of 
$l$  split equations on 
any $n$ variables
 $$ w_i(x_1, ..., x_n) = h_i, h_i \in H,  i = 1, ..., l,$$
 \noindent
  that has a solution in $G,$ is solvable in $H$.

\end{definition}
In the case $l=1$ this notion means verbal closeness.

\begin{remark}
\label{re:2}
Note that  the verbal closeness does not imply  $l$-verbal closeness of  $H$ in $G$. Indeed, every equation $w(x_1, ..., x_n; H)=h, h\in H,$ is equivalent to a system of split equations with right sides in $H$. This system can be obtained from the initial equation if we change each coefficient $h$ to new variable $x_h$ and add the corresponding split equation $x_h=h$ to the system.  Hence any non-algebraically closed subgroup $H\leq G$ is non-l-verbally closed for some $l.$ We noted  in Remark  \ref{re:1}, that there are verbally closed but non-algebraically closed  subgroups of some groups.
\end{remark}

Further in the paper we denote a conjugate by $g^f=fgf^{-1}$ , and commutator by $[f,g] = fgf^{-1}g^{-1}$. Here  $f, g$ are elements of arbitrary group  $G.$ By $G'$ we denote the derived subgroup (commutant) of $G$. By $\gamma_iG$ we mean the $i$-th member of the low central series in $G$    (in particular, $\gamma_1G=G, \gamma_2G=G'$). By  $G^{(i)}$ we denote the $i$-th member of the derived series in $G$ (in particular,  $G^{(0)}=G,  G^{(1)}=G'$).
By  $\langle g_1, ..., g_k\rangle $ we denote  the subgroup generated by elements $g_1, ..., g_k$ in  $G.$ The variety of all abelian grous is denoted  as $\mathfrak{A}$, the variety of all nilpotent groups of nilpotency class $c$ ($c \in \mathbb{N}$) is denoted as  $\mathfrak{N}_c$. By $\mathfrak{A}^d$ ($d \in \mathbb{N}$) we mean the variety of all solvable groups of the solvability class  $\leq d.$

 For any tuple of varieties  $\mathfrak{B}_i, i = 1, ..., k, \,
\prod_{i=1}^k\mathfrak{B}_i$ denotes their product. We write $F_r(\mathfrak{B})$ to denote a free group of rank  $r$ in $\mathfrak{B}.$ Further in the paper we assume that any rank $r$ is finite. For some relatively free groups we use specific denotions: $A_r=F_r(\mathfrak{A}), \, N_{r,c}=F_r(\mathfrak{N}_{r,c}),\, S_{r,d} = F_r(\mathfrak{A}^d).$

 Let $H$ be a finitely generated subgroup of a group   $G.$
Then $r_{ab}(H)$ denotes a rank, i.e., the minimal number of generators 
of $HG'/G'.$

If $H$  is verbally closed then  $G$ $\cap G' =H'.$  Indeed, if $h\in G'\cap H$, then   some equation of the form $h=\prod_{j=1}^t[x_j, x_j']$ is solvable in  $G$.   Then this equation has a solution in  $H$, i.e., $h\in H'.$ Then  $H/H'
\cong H/H\cap G' \cong HG'/G'.$

The following assertion is followed from the statements above.

 \begin{proposition}
 \label{pr:1}
 {\it Let $G=\langle g_1, ..., g_k, ...\rangle $ be a solvable group and 
$H$ be a subgroup of $G$ such that  $r_{ab}(H) = 0,$ i.e.,  $H\leq G'.$ If  $H$ is verbally closed in  $G,$ then $H=1$. }
\end{proposition}
\begin{proof}
 Suppose that  $h\in H, h \not= 1.$ 
 Let $h\in G^{(n)}\setminus G^{(n+1)}, n\geq 1.$  Then  $h$ can be expressed in the form
$h = \prod_{i=1}^mh_i,$ where every factor $h_i$ is written as
$[v_i, w_i],\,v_i, w_i\in G^{(n-1)}.$  Let $v_i=v_i(g_1, ..., g_r)$ and  $w_i=w_i(g_1, ..., g_r)$ are  the corresponding expressions of the elements  $v_i, w_i$ in generators of  $G$. The words  $v_i(x_1, ..., x_n)$ and  $w_i(x_1, ..., x_n)$  are chosen in  $F_n^{(n-1)}$,  where $F_n$ is the free group with  basis   $\{x_1, ..., x_n\}.$ Hence, the equation  $\prod_{i=1}^m [v_i(x_1,\ldots,x_r),
w_i(x_1,\ldots,x_r)]= h$ has a solution in $G.$  Obviously this equation has no solutions in  $G'$ because any substitution of elements from $G'$ instead of variables gives  an element in  $G^{(n+1)}.$
Especially, this equation is not solvable in  $H.$ 

\end{proof}

The following statements are main results of this paper.

\begin{theorem}
\label{th:1}   Let $H$ be a finitely generated verbally closed subgroup of a free solvable group   $S_{r,d}$.  If $r_{ab}(H) = 1,$ then $H$ is a retract.
\end{theorem}

\begin{theorem}
\label{th:2} Let $H$ be a finitely generated verbally closed subgroup of a free solvable group 
 $S_{r,d}$.  If
$r_{ab}(H) = r, $ then $H=S_{r,d}.$
\end{theorem}

\begin{theorem}
\label{th:3} Every $2$-generated  verbally closed subgroup $H$ of a free solvable group 
    $S_{r,d}$ is a retract.
\end{theorem}

\begin{theorem}
\label{th:4} For any $l\in \mathbb{N},$ every finitely generated 
$l$-verbally closed subgroup  $H$ of a free metabelian group 
  $M_r$ such that $r_{ab}(H)=l+1$ is a retract.
\end{theorem}

\bigskip

\begin{center}

{\bf \S1. PRELIMINARY RESULTS}

\end{center}

\bigskip

Recall that a set of elements $\{g_1,\ldots,g_n\}$ of a group $G$
is  {\it test set} if any endomorphism  $\varphi : G
\rightarrow G,$ that fixes every element  $g_i\  (i=1, ..., n),$  is an automorphism of 
 $G.$ An element  $g \in G$ is said to be  {\it
test element} if  $\{g\}$ is test set. Test rank $tr(G)$ is defined as the minimal number of elements in a test set of $G.$ See \cite{Timmet}, \cite{Romtest}, \cite{Timsolv}, \cite{Timpol}, \cite{Timbook} for details and results about test sets and test ranks of groups.

Let $F_r$ be a free group of rank  $r$ and let $V$ be a normal subgroup of $F_r.$
A test set    $\{g_1,\ldots, g_n\}\subseteq G=F_r/V'$
is said to be  {\it strongly test set} if any endomorphism 
$\varphi : G \rightarrow G$ fixing every element $g_i$ is inner automorphism that  acts as  conjugate by element from 
$V/V'.$

Recall, that a variety of the form $\mathfrak N_{c_1} \ldots
\mathfrak N_{c_m}, m \geq 1$ is called  {\it polynilpotent of class
 } $\bar{c}=(c_1, ..., c_m).$ We consider a polynilpotent variety of the form 
  P$\mathfrak{N}_{\bar{c}}=\mathfrak A \mathfrak
N_{c_1} \ldots \mathfrak N_{c_m}, \,m \geq 1,$ of class $\bar{c}=(2,
c_1, ..., c_m)$. Note that class $\bar{c}=(2, 2, ..., 2)$
corresponds to  $\frak{A}^{m+1}.$  The following lemma follows 
from results and their proofs of the papers (\cite{Timsolv},
\cite{Timpol}) by C.K. Gupta and the second author. 
\begin{lemma}
\label{le:1}
A relatively free group  $F_{r}($P$\mathfrak{N}_{\bar{c}})$ of rank $r \geq 2$ of the variety
P$\mathfrak{N}_{\bar{c}}, \,m \geq
1,$ admits a strongly test set consisting of $(r-1)$ elements.
\end{lemma}

\begin{proof}  Let  $F_r/V$ be a free group of rank  
 $r$ of the polynilpotent variety $\mathfrak N_{c_1} \ldots
\mathfrak N_{c_m}, m \geq 1.$ The main statement in
\cite{Timsolv} is computing of the test rank of any free solvable group.  In fact, in this paper a test rank of  $F_r/V'$ was compute. Especially, it was shown that the test rank of a free group 
 of the polynilpotent  variety of the form  $\mathfrak A \mathfrak N_{c_1}
\ldots \mathfrak N_{c_m}$ was computed.  Moreover, it was shown how a test set 
$\{g_1, \ldots, g_{r-1}\}$ for $F_r/V'$ can be constructed such that: if an endomorphism $\varphi$
fixes every element of this set  then 
$\varphi$ is an inner automorphism that induces identity map on  $V/V'.$ In \cite{Timpol}, it was shown that the mentioned above inner automorphism is induced by conjugate to an element that lies in  $V/V',$ then
  $\{g_1, \ldots, g_{r-1}\}$ is  strong test set.

\end{proof}

\begin{lemma}
\label{le:2}
 Let $G = F_r/V' = F_r(\mathfrak A \mathfrak
N_{c_1} \ldots \mathfrak N_{c_m}), \,m \geq 1, \,r \geq
2,\,$    and let $\{z_1,\ldots, z_r\}$ be a basis of  $G.$   Let $H \leq G$ is isomorphic to 
 $F_n(\mathfrak A \mathfrak N_{c_1} \ldots \mathfrak
N_{c_m})$ for some  $n \geq 2.$ Let $T =
\{t_1(z_1,\ldots,z_r), \ldots, t_{n-1}(z_1,\ldots, z_r)\}$ be a strong test set in 
 $H.$ Suppose that 
$$t_i(h_1,\ldots,h_r) = t_i(z_1,\ldots, z_r),\,\, \,\,i =1,\ldots,n-1,$$ for some elements $h_1,\ldots, h_r  \in H.$ Then map 
$$z_j \mapsto h_j,\,\,j=1,\ldots,r,$$ is extended to a retraction 
 $G \rightarrow H.$
\end{lemma}
\begin{proof}
 Extend map $z_j \mapsto
h_j,\,\,j=1,\ldots,r,$ to endomorphism $\varphi : G
\rightarrow H.$ Let $\varphi|_H$ be restriction of $\varphi$ to
$H, $ that is an endomorphism of  $H.$

Then $\varphi|_H(t_i(z_1,\ldots,z_r)) = \varphi
(t_i(z_1,\ldots,z_r))= t_i(\varphi(z_1),\ldots, \varphi(z_r)) =
t_i(h_1,\ldots,h_r) = t_i(z_1,\ldots, z_r), \,i=1,\ldots,n-1.$
Hence $\varphi|_H$ is an inner automorphism   $\alpha_a$
of $H$ induced by  $a \in V/V'.$ Consider an inner automorphism  
$\psi = \varphi_{a^{-1}} \varphi,$ where $\varphi_{a^{-1}}$ 
of $G$ corresponds to  $a^{-1}.$ For
$h \in H$ we have $$\psi (h) = \varphi_{a^{-1}} \varphi (h)=
\varphi_{a^{-1}}  (h^a) = h,$$ and for $g \in G$ $$\psi (g) =
(\varphi (g))^{a^{-1}} \in H.$$ Thus $\psi$ is a seeking retraction onto  
$H$.

\end{proof}

The lemmas above imply the following assertion.

\begin{proposition}
\label{pr:2} {\it Let  $H$ be a subgroup of $F_r(\mathfrak A
\mathfrak N_{c_1} \ldots \mathfrak N_{c_l})$  generated by $(m+1)$
elements and   rank of  $H_{ab} = H/H'$ is equal to  $m+1.$  Then, if 
$H$ is  $m$-verbally closed subgroup, it is a retract.}

\end{proposition}

\begin{lemma}
\label{le:3}
(\cite{RomKhis}).  Let  $N$ be a verbal subgroup of a group $G.$ If  $H$ is a verbally closed subgroup of   $G,$ then its image $H_N = HN/N$ is verbally closed in  $G_N =
G/N.$ Specifically, if $G/G'$  is a free abelian group of a finite rank then the image $\bar{H}$  of its verbally closed subgroup  $H$ in  $G/G'$  is a direct factor.
\end{lemma}

\begin{lemma}
\label{le:4}
  Let $G=S_{r,d}$ be a free solvable group of class
 $d \geq 2$   with basis $z_1, \ldots, z_r, \,\,r \geq
2.$ Let $H\leq G$ be a verbally closed subgroup generated by  elements $c_1z_1, \ldots, c_rz_r, c_{r+1} \ldots,$ where any  $c_i$ lies in 
 $G^{(d-1)}.$ Then $H^{(d-1)}$ is $\mathbb
Z[H/H^{(d-1)}]-$closed subgroup of  $G^{(d-1)},$ i.e.,  if for any 
 $c \in G^{(d-1)}$ and $0 \neq \alpha \in
\mathbb Z[G/G^{(d-1)}]$ we have $c^\alpha \in
H^{(d-1)},$ then $c \in H^{(d-1)}.$
\end{lemma}

\begin{proof} Denote $ A = G/G^{(d-1)},\,\, B =
H/H^{(d-1)},$  then $G^{(d-1)} \cap H = H^{(d-1)}.$

Let $1 \neq c^\alpha \in H^{(d-1)}.$ Write $c$ as  $c = c(z_1, \ldots, z_r).$ Then $\alpha$ has an expression of the form  $\alpha = \Sigma m_p\overline{h}_p,$ where $m_p
\in \mathbb Z,$ $\overline{h}_p \in B.$ Let $a_i =
z_iG^{(d-1)}.$ Then every element  $\overline{h}_p$
lies in the subgroup generated by   $a_1,\ldots, a_r,$
i.e.,  $\alpha = \alpha(a_1, \ldots, a_r).$

Consider an equation  in unknowns $x_1,\ldots,x_r:$
$$c(x_1,\ldots, x_r)^{(1-\overline{x}_1^{m_1})(1-\overline{x}_1^{-m_1})\ldots
(1-\overline{x}_r^{m_r})(1-\overline{x}_r^{m_r})\alpha(\overline{x}_1,
\ldots,\overline{x}_r)}$$ $$ = c(z_1,\ldots,
z_r)^{(1-a_1^{m_1})(1-a_1^{-m_1})\ldots
(1-a_r^{m_r})(1-a_r^{-m_n})\alpha(a_1, \ldots,a_r)},
$$ where $m_i$ are integers. This equation has a solution in  $G.$ Hence, it has a solution 
 $x_i = h_i$ in $H.$ Then
\begin{equation}
\label{eq:2}
c(h_1,\ldots, h_r)^{(1-\overline{h}_1^{m_1})(1-\overline{h}_1^{-m_1})\ldots
(1-\overline{h}_r^{m_r})(1-\overline{h}_r^{-m_r})\alpha(\overline{h}_1,
\ldots,\overline{h}_r)}=
\end{equation}
 $$ = c(z_1,\ldots,
z_r)^{(1-a_1^{m_1})(1-a_1^{-m_1})\ldots
(1-a_r^{m_r})(1-a_r^{-m_r})\alpha(a_1, \ldots,a_r)}.
$$ Here  $\overline{g}$  denotes the image of  \,  $g \in G$  in $A.$

 The group ring  $\mathbb Z[G]$ is equipped by the left Fox 
derivations $\partial_j(g)$ with values in 
 $\mathbb Z[A]$ (see details in  \cite{Timbook}, \cite{RemSoc},
\cite{KrFox}, \cite{Romess}).
 For $\alpha
\in \mathbb Z[A]$ and $b \in G^{(d-1)}$ the following equality is valid:
$$\partial_i(b^\alpha) = \alpha \partial_i(b).$$ Computing the values of $j$th derivation of two sides of the equality 
 (\ref{eq:2}) we obtain
\begin{equation}
\label{eq:3}
(1-\overline{h}_1^{m_1})(1-\overline{h}_1^{-m_1})\ldots
(1-\overline{h}_r^{m_r})(1-\overline{h}_r^{-m_r})
\alpha(\overline{h}_1,
\ldots,\overline{h}_r)\partial_j(c(h_1,\ldots, h_r)) =
\end{equation}
 $$= (1-a_1^{m_1})(1-a_1^{-m_1})\ldots
(1-a_r^{m_r})(1-a_n^{-m_r})\alpha(a_1, \ldots,a_r)\partial_j(c).$$

Let $\Delta$ be  the fundamental ideal of  $\mathbb Z[A].$ 
$\Delta$  is endowed with a norm  $\omega$   (see \cite{Romanovski}). Namely, for an element 
 $u \in \Delta$ we set $\omega(u) = n,$ where
$u \in \Delta^n \backslash \Delta^{n+1}.$

We'll show that no  $h_i$ lies in $G'.$
Since $c = c(z_1,\ldots, z_r)\not= 1,$  there is $j$ for which $\partial_j(c) \neq 0.$
 We choose $j$, for which the norm of   $\partial_j(c)$ is minimal among all 
$\omega(\partial_1(c)), \ldots, \omega(\partial_r(c)).$ We can assume for simplicity that  $j=1.$ In  (3) we compare the norms of the left and the right sides.  By the known rule we have
$$\partial_1(c(h_1,\ldots,h_r)) =
\partial_1(c)[\overline{h}_1,\ldots,\overline{h}_r]\partial_1(h_1) + \ldots +
\partial_r(c)[\overline{h}_1,\ldots,\overline{h}_r]\partial_1(h_r),
$$ where $\partial_i(c)[\overline{h}_1,\ldots,\overline{h}_r]$
means a result of substitution to  $\partial_i(c)$
of elements  $\overline{h}_q$ instead $\overline{z}_q.$  By choice of $j$
we obtain  $\omega(\partial_1c(h_1,...,h_r)) \leq
\omega(\partial_1(c)).$ If at least one element  $h_i$ belongs to
$G',$ the norm of the left side of  (3) is less than the norm of the right side, because
$\omega(1-a_i) = 1 < \omega( 1 - \overline{h}_i),$ for $h_i \in
G'.$

In \cite{Timsolv},   the following statement was proved.
 Let $S$ be a free polynilpotent group of rank 
$r \geq 2$ with basis $\{s_1,...,s_r \},$ and $0\not= \alpha \in \mathbb Z [S].$ Let  $p_i $ be the least
$s_i-$ exponent of  $\alpha ,$  and $q_i$ be the greatest $s_i -$
exponent of it. Let a positive integer  $m$ is greater than 
$$ \sum_{i=1}^r (|p_i|+|q_i|).$$ For $y_1\in G\setminus S',$ equality
\begin{equation}
\label{eq:4}
(1-s_1^m)\alpha
=(1-y_1^m)\beta
\end{equation}
\noindent
is possible for $\beta \in
\mathbb Z [S]$ if and only if  $y_1=s_1$ or
$y_1=s_1^{-1}.$

\bigskip

Now we choose an available $m_i$. Then $\overline{h}_1 =
a_1^{\pm 1}.$ The ring  $\mathbb Z[B]$ has no zero divisors. The expression 
 $(1 - a_i^{m_i})(1-a_i^{-m_i})$ is not changed under
substitution  $a_i$ to $a_i^{-1}.$ We cancel step by step  (\ref{eq:3})
to $(1 - a_i^{m_1})(1 - a_i^{-m_1})$ and obtain 
 $\overline{h}_i = a_i^{\pm 1}$ for all $i = 1,\ldots,r.$
Then we get $\partial_i(c(h_1,\ldots,
h_r)) =
\partial_i(c),$   and
$$  c = c(h_1,\ldots,h_r) \in H^{(d-1)}.$$

 \end{proof}

\begin{lemma}
\label{le:5}
Let $H = \langle h \rangle$ be a non-trivial cyclic subgroup of  $G=S_{r,d}.$
 Then the following conditions are equivalent:
\begin{enumerate}
\item
$H$ is a verbally closed in $G;$
\item
$H$ is a retract of  $G;$
\item
the image $\bar{h}$ of $h$ in $G/G',$ i.e., in the free abelian group   $A_r,$ is primitive. 
\end{enumerate}
\end{lemma}

\begin{proof} By Lemma  \ref{le:3}   1 implies  3.

Let we show that  3 implies 2. Write $h$ as $h = au,$ where $a$ is a primitive element of 
 $A_r,\,\,  u \in G'.$ Since any basis 
of $A_r$ is induced by a basis of  $G$ (moreover, it is induced by a basis of  $F_r,$ see for example
\cite{Romess}). Hence there is a preimage  $z$ of $a$ in $G$ that belongs to a basis of  $G.$ We take the endomorphism of $G$  that maps  $z$ to  $,$
and any other  element of this basis  to   $1.$ We see that this endomorphism is a retraction 
  $G\rightarrow H.$

By Remark \ref{re:1},   2 implies 1.

\end{proof}

\begin{lemma}
\label{le:6} Let  $H = \langle g, f \rangle$ be a  $2$-generated
non-cyclic subgroup of a free solvable group $G =S_{r,d}$
of rank $r \geq  2$.  Then the following statement is valid: if $H$ is a verbally closed subgroup of 
 $G,$ then $g, f$ induce in $A_r =
G/G'$ (free abelian group of rank  $r$) a direct factor of rank 
2, and $H$ is a free solvable group of rank  2
and class $d.$
\end{lemma}
\begin{proof}  By Lemma \ref{le:3}, the images of  $g, f$  in
 $A_r$ generate a verbally closed subgroup,  that is a direct factor of full group.
It remains to prove, that this factor is non-cyclic. 
  By Lemma  \ref{le:3} and induction on  $d$ we assume that the image of  $H$ in $S =G/G^{(d-1)}$ is generated by the image $\overline{g}$ of $g$, and $f$ lies in $G^{(d-1)}.$  Let $f (z_1,…,z_r)$ be an expression 
of  $f$ as a word in generators of  $G.$ Equation $f (x_1,…,x_r) = f$ is solvable in  $G.$ Hence, it has a solution   $h_1,…,h_r$ in $H.$
The components of this solution are written as  $h_i = g^{t_i}
f^{\alpha_i}, t_i \in \mathbb Z,$ where $ \alpha_i\in \mathbb Z[\overline{g}],\, i = 1, …, r.$ Then $f
(h_1,…,h_r) = f (g^{t_1},…,g^{t_r}) f^\delta,$ where
$\delta\in \mathbb Z[\overline{g}]$ is easily to compute by collecting  exponents of $f$ in $f (h_1,…,h_r).$ It remains to note that $\delta$ lies in the fundamental ideal of 
$\mathbb Z[\overline{g}],$ because $f (x_1,…,x_r)$ is a commutator word.
 Then $f^{1-\delta} = 1$ for $1 - \delta \not= 0$  that is impossible in view of well-known fact that  
$G^{(d-1)}$ has no module torsion (see for example  \cite{Romess}). By Baumslag's theorem  \cite{B} $H = \langle g, f \rangle$ is a free solvable group of rank $r$ and class  $d$ with basis consisting of the elements described above.

\end{proof}

The following statement was proved in  \cite{Romeqmet} (see also \cite{Romess})

\begin{lemma} (\cite{Romeqmet}).
\label{le:7}
Let  $M_2$ be the free metabelian group of rank 
 $2$ with  basis $\{z_1, z_2\}.$ Let $x_1=g_1, x_2=g_2$  be
a solution of  $$ [x_1, x_2, x_2, x_1] \equiv  [z_1, z_2, z_2,
z_1].$$ Then $g_i\equiv  z_i^{\varepsilon_i}(mod\  M_2'),
\varepsilon_i\in \{\pm 1\}.$
\end{lemma}

Here and further  $[g_1, g_2, ..., g_k]$ means a simple 
commutator that is defined inductively:  $[g_1, g_2, ... ,
g_k]=[[g_1, g_2, ... , g_{k-1}], g_k].$

\begin{center}

{\bf \S2. PROOFS OF THE MAIN RESULTS}

\end{center}

{\bf Proof of Theorem   \ref{th:1}.} Denote $G =S_{r,d}.$
We use induction on  $d.$ The case  $d=1$ is trivial.
Suppose that $r, d \geq 2.$ By induction  and Lemma  \ref{le:3} we can choose a basis 
$\{z_1,\ldots,z_r\}$ of $G$ in such way that $$H = \langle z_1 c_1,
c_2, \ldots, c_m \rangle,$$  where $c_1,\ldots,c_m \in G^{(d-1)}.$
The quotient $G/G^{(d-1)}$  is isomorphic to $S_{r,d-1}.$ Denote $C
= G^{(d-1)}$ and $A = G/C.$ Let $\{a_1, ..., a_r\}$ be a basis of 
 $A$ corresponding to $\{z_1, ..., z_r\}.$

 We'll show that each  $c_i$  is equal to $1.$
We write any $c_i$ in the form
$c_i = c_i(z_1,\ldots,z_r),$ where  $c_i(x_1,x_2,\ldots, x_r)$ is a word that lies in the  $(d-1)$th member of the derived series of  the free group $F(X_r), X_r = \{x_1, ..., x_r\}$. Since $H$ is verbally closed
there exist elements $h_1, ..., h_r \in H,$ such that
\begin{equation}
\label{eq:5}
c_i = c_i(h_1, \ldots, h_r).
\end{equation}
 From 
(\ref{eq:5}) we get the equalities:
\begin{equation}
\label{eq:6}
 c_i = {c_1}^{-\delta_{i1}} \ldots {c_i}^{-\delta_{ii}}\ldots {c_m}^{-\delta_{im}}, i=1, ..., m,
 \end{equation} where $ \delta_{ij}(a_1)\in \mathbb Z[a_1^{\pm 1}].$  Let
$\nu_{i} = 1 + \delta_{ii}.$ Each element 
$\delta_{ij}$  for $i\not= j$  lies in 
ideal $id(a_1-1)$ of $\mathbb{Z}[a_1^{\pm 1}],$  and any  $\nu_{i}=1$  modulo this ideal.

Then  (\ref{eq:6}) implies a system of equalities:

$$c_1^{\nu_1} c_2^{\delta_{12}} \ldots c_m^{\delta_{1m}} =1,$$
$$c_1^{\delta_{21}} c_2^{\nu_2} \ldots c_m^{\delta_{2m}} =1,$$
$$ \ldots  \ldots  \ldots  \ldots $$
$$c_1^{\delta_{m1}} c_2^{\delta_{m2}} \ldots c_m^{\nu_m} =1.$$
Compute Fox derivations of the left sides that are equal to $0.$ 
$$
\nu_1\partial_j(c_1) + \delta_{12}\partial_j(c_2) + \ldots +
\delta_{1m}\partial_j(c_m)= 0, $$
$$\delta_{21}\partial_j(c_1) + \nu_2\partial_j(c_2) + \ldots +
\delta_{2m}\partial_j(c_m)= 0,$$
\begin{equation}
\label{eq:7}
 \ldots  \ldots  \ldots  \ldots
 \end{equation}
$$\delta_{m1}\partial_j(c_1) + \delta_{m2}\partial_j(c_2) + \ldots +
\nu_{m}\partial_j(c_m)= 0, $$
\noindent
 where $j = 1, \ldots, m.$ Determinant of each system  (\ref{eq:7}) is $1$ modulo  $id(a_1 -
1),$ thus it is not  $0.$ Hence the system has the only zero solution.  It follows that each  $c_i$ is $1.$  By Lemma  \ref{le:5} we finish the proof.

$\square$

{\bf Proof of Theorem \ref{th:2}}).   Denote $G=S_{r,d}$.
We use induction on  $d.$ The case  $d = 1$ is trivial.
We assume that $r,d \geq 2.$ By  induction  and Lemma 
 \ref{le:3} we construct a basis  $\{z_1,\ldots,z_r\}$ of $G$
such that $$H = \langle z_1c_1, \ldots, z_rc_r,
c_{r+1},...,c_m \rangle,$$ where $c_1,\ldots,c_r,
c_{r+1},...,c_m \in C.$ It is enough to check that each element 
$c \in C$ lies $H^{(n-1)}.$

Denote  $C=G^{(d-1)}, A = G/C \simeq S_{r,d-1}$.
  Abelian group $C$ is a right  $\mathbb{Z}[A]-$module in which action of  $a\in A$ to $c \in C$ is defined as $a^{-1}ca.$
  Elements
$a_i=z_iC,\,\, i=1,\ldots,r,$ give a basis of $A.$  Let 
$T$ be a free right  $\mathbb{Z}[A]-$module with basis
$\{t_1,\ldots, t_r\}.$ We take the Magnus embedding of  $G$ into split extension
 $$M = \left(\begin{array} {ll}
A & 0 \\
T & 1
\end{array}\right).$$
See details in  \cite{Timbook}, \cite{Romess}. Any element
  $z_i$ maps to  $$\left(\begin{array} {ll}
a_i & 0 \\
t_i & 1
\end{array}\right).$$
The image of  $C$ in  $M$
is a submodule of 
  $T,$ that is identified with $C.$ An element  $t_1 u_1+ \ldots t_ru_r,\,\,u_i \in \mathbb Z[A],$
lies in  $C$ if and only if 
\begin{equation}
\label{eq:8}
u_1(a_1-1) + \ldots
+u_r(a_r-1) = 0.
\end{equation}
The image of   $C$ in  $M$
is a submodule of 
  $T,$ consisting of all elements that satisfy to  (\ref{eq:8}). We keep its denotion  $C.$

 It is well-known that  $ \mathbb{Z}[A]$ satisfies to the right Ore condition 
 \cite{L}, i.e., for any pair of non-zero elements elements
$\alpha , \beta \in\mathbb{Z}[A]$ there are non-zero elements  $\mu ,\nu \in \mathbb{Z}[A]$ such that  $\alpha \mu= \beta \nu  \neq 0.$ Since $\mathbb{Z}[A]$ is also a domain
it is embedded into a sfield of fractions $P.$ Then  $T$
embeds into  a linear space  $V$ of dimension  $r$ over
$P.$ By (\ref{eq:8}) we conclude that the dimension of the  linear subspace
$C^P$ generated by  $C$ in $V$ is $r-1.$

Let we consider a subset of non-trivial elements  \{$w_i = w_i(z_1,
z_i),\,\,i=2,\ldots,r$\} of $C.$ It is easily compute that the image of   $w_i$ under Magnus embedding   in  $M$
(more exactly in $T$) has the form $t_1\gamma_1 + t_i\gamma_i,$
where  $\gamma_1$ and $\gamma_i$ are non-zero. Independent of  $t_1\gamma_1 + t_i\gamma_i, \,i=2,\ldots,r$ in $T$
is clear. We are to check that elements   $w_i(z_1c_1,
z_ic_i),\,i=2,\ldots,r,$ are linearly independent over  $\mathbb
Z[A].$

Denote by $G_1$ the subgroup generated by elements  
$z_1c_1,\ldots, z_rc_r.$ By Baumslag's theorem \cite{B} the map
$$\varphi : z_i \mapsto z_ic_i, \,\,\,i= 1,\ldots,r,$$ extends to isomorphism  $G$ and $G_1,$ and then  to isomorphism of  $C$
and $G_1^{(d-1)}$, and also to isomorphism of $\mathbb Z[A]$ and $\mathbb
Z[G_1/G_1^{(d-1)}].$ Moreover,   $\mathbb Z[A]-$module $C$
is isomorphic to  $\mathbb Z[G_1^{(d-1)}]-$module $G_1^{(d-1)}.$  Any of these isomorphisms is denoted by  $\varphi .$

  We are to check that elements  $ w_i(z_1c_1, z_ic_i),\,\,i= 2,\ldots, r$ are independent over 
$\mathbb Z[A].$

 Suppose that it is not happen:
 $$w_2(z_1c_1, z_2c_2)^{\beta_2(z_1,\ldots, z_r)} \ldots
w_r(z_1c_1, z_rc_r)^{\beta_r(z_1,\ldots, z_r)} =1$$ for some
$\beta_i \in \mathbb Z[A],$ at least one of them is not $0$.
Hence $$w_2(z_1c_1, z_2c_2)^{\beta_2(z_1c_1,\ldots,
z_rc_r)} \ldots w_r(z_1c_1, z_rc_r)^{\beta_r(z_1c_1,\ldots,
z_rc_r)} =1.$$  We apply  $\varphi^{-1}$ and get a contradiction 
with independence of  $w_2,\ldots,w_r$ over
$\mathbb Z[A].$

We are to check that elements  $ w_i(z_1c_1, z_ic_i),\,\,i=
2,\ldots, r,$ are linearly independent not only over  $\mathbb Z[A],$ but over 
$P$  too.
 Elements of  $P$ are equivalency classes of  $\mathbb Z[A] \times \mathbb Z[A],$
 that usually denoted as  $\alpha \beta^{-1}$ for $ 0 \neq
 \beta,\alpha\in \mathbb Z[A].$ Suppose that there is a non-trivial linear combination over $P$:
 $$w_2(z_1c_1, z_2c_2)\alpha_2\delta_2^{-1} + \ldots + w_r(z_1c_1,
 z_rc_r)\alpha_r\delta_r^{-1}= 0.$$
Since   $\mathbb Z[A]$ is the Ore domain every a finite set of its elements has a right common factor.  Take such factor $\delta $ for  $\delta_2,\ldots, \delta_r.$  By definition $0 \neq \delta =
\delta_i\gamma_i,\,\,i=2,\ldots,r, \gamma_i \in \mathbb Z[A].$
Then
$$w_2(z_1c_1, z_2c_2) {\alpha_2 \gamma_2} + \cdot + w_r(z_1c_1,
z_rc_r){\alpha_r\gamma_r} =0.$$ This contradicts to linear independence of elements $ w_i(z_1c_1, z_ic_i),\,\,i= 2,\ldots, r,$ over $\mathbb
Z[A].$

 We proved that elements $w_2(z_1c_1,z_2c_2), \ldots, w_r(z_1c_1,z_rc_r)$
are linearly independent over $P.$

There are elements $0 \neq \sigma, \sigma_1, \ldots,
\sigma_{r-1}\in P$ such that
$$c^{\sigma} = w_2(z_1c_1, z_2c_2)^{\sigma_1} \ldots
w_r(z_1c_1, z_r c_r)^{\sigma_{r-1}}.$$ By multiplying all 
 $\sigma, \sigma_1, \ldots, \sigma_{r-1}$ to available element  
$u \in \mathbb Z[A]$, we can assume that $\sigma,
\sigma_1, \ldots, \sigma_{r-1} \in \mathbb Z[A].$ Thus $c^\sigma
\in H^{(d-1)}.$ Lemma \ref{le:4} implies that  $c \in
H^{(d-1)}.$

$\square$

{\bf Proof of Theorem   \ref{th:3}.}
 By Lemma \ref{le:6}   $H$ is  solvable of rank $2.$ By Baumslag's theorem
 \cite{B} it isomorphic to $S_{2d}.$ By Lemma
\ref{le:1} $H$ contains a strong test element. By Lemma  \ref{le:2} 
$H$ is a retract.

$\square$

{\bf Proof of Theorem   \ref{th:4}.} Denote $G = M_r.$
Let $\{z_1, ..., z_r\}$  be a basis of   $G$, and let $\{a_1, ...,
a_r\}$ be the corresponding basis of $A_r=G/G'.$

By Lemma \ref{le:3} we can assume that  $H=\langle h_1, ..., h_{l+1}, h_{l+2}, ... h_{l+t}\rangle$, where
  $h_i\equiv z_i (mod\  G'),\  i = 1, ..., l+1,$ and $h_{l+k}\in G', k =2, ..., t.$

  The case $l=0$ was proved in Theorem \ref{th:1}.
 Suppose that $l \geq 1.$

   For any $g \in G$  by $\bar{g}$\,  we denote its image in  $A_r.$ Then $\bar{h}_i = a_i,\  i
= 1, ..., l+1; \bar{h}_{l+k}=1,\  k= 2, ..., t.$  Hence $\bar{H} = A_{l+1}= \langle a_1, ..., a_{l+1}\rangle .$ By Baumslag's theorem  \cite{B} subgroup $H_{l+1} = \langle h_1, ..., h_{l+1} \rangle$
is isomorphic to $M_{l+1}$, and $\{h_1, ..., h_{l+1}\}$  is a basis of
 $H_{l+1}$. In particular, every map of this basis to $G$ extends to a homomorphism  $H_{l+1} \rightarrow G.$

In the proof that follows we show that there is a homomorphism   $\psi : G \rightarrow H$ identical on $H_{l+1},$ and then  we show that it is identical on  $H,$ i.e., it is a retraction of $G$ onto $H.$

Let $h_i = h_i(z_1, ..., z_r), i = 1, ..., l+1$ be expressions fixed generators of
 $H_{l+1}$ as words in basic elements of 
 $G$. For any $i\in \{2, ..., l+1\}$  we consider equation
\begin{equation}
\label{eq:15}
 [h_i(x_1, ..., x_r), h_1(x_1, ..., x_r), h_1(x_1, ..., x_r),
h_i(x_1, ..., x_r)] = [h_i, h_1, h_1, h_i].
\end{equation}
This equation is solvable in  $G$: $x_j=z_j, j = 1, ...,
r.$ Hence it has a solution in  $H$: $x_j = f_j, f_j \in H, j = 1,
..., r.$

We have:
$$ [h_i(f_1, ..., f_r), h_1(f_1, ..., f_r), h_1(f_1, ..., f_r),
h_i(f_1, ..., f_r)] \equiv [h_i, h_1, h_1, h_i](mod \
\gamma_5G),$$
\begin{equation}
\label{eq:16}
 i=2, ..., l+1.
\end{equation}

Denote $\tilde{h}_j = h_j(f_1, ..., f_r) , \ j =1, i.$ Let $\tilde{h}_1\equiv \prod_{j=1}^{l+1}z_j^{k_j} (mod \,  G'),
\tilde{h}_i\equiv \prod_{j=1}^{l+1}z_j^{m_j} (mod \, G'), \ k_j, m_j \in \mathbb{Z}, \ j =1, ..., l+1.$  Then (\ref{eq:16})  is equivalent to

\begin{equation}
\label{eq:17}
 [\prod_{j=1}^{l+1}z_j^{m_j},\prod_{j=1}^{l+1}z_j^{k_j}, \prod_{j=1}^{l+1}z_j^{k_j},\prod_{j=1}^{l+1}z_j^{m_j}] \equiv [z_i, z_1, z_1, z_i](mod \
\gamma_5G).
\end{equation}
Quotient $G_5=G/\gamma_5G$ is a free group of variety $\mathfrak{A}^2\cap \mathfrak{N}_5$ of all metabelian nilpotent groups of nilpotency class $\leq 5$. Its the last non-trivial member of the low central series   $\gamma_4G_5$ is a free abelian group with basis consisting of all basic commutators of length $4$ of the form   $[z_{i_1}, z_{i_2}, z_{i_3}, z_{i_4}],$ where $i_1>i_2, i_2\leq i_3 \leq i_4.$ Any metabelian group satisfies to identities 
 $[x_1, x_2,x_3, ..., x_n]=[x_1, x_2, x_{\tau (3)}, ..., x_{\tau (n)}]$, where $n\geq 3$, $\tau $ is a substitution on  $3, ..., n.$  $G_5$ satisfies to identity $[x_1^{k_1}, x_2^{k_2}, x_3^{k_3}, x_4^{k_4}] = [x_1, x_2, x_3,x_4]^{k},\, k =\prod_{j=1}^4k_j.$ See details in  \cite{Romess}.

These facts mean that after rewriting of the left side of  (\ref{eq:17})  as a product of exponents of basic commutators  of length  $4$ these exponents will be  $0$ with just one exception for exponent of   $[z_i, z_1, z_1, z_i]$, which in correspondence with the right side of  (\ref{eq:17}) should be   $1.$ We write a product of all basic commutators  in this expression that depend of  $z_1$ and $z_i$ only.
\begin{equation}
\label{eq:18}
[z_i,z_1, z_1, z_i]^{\alpha_1}[z_i, z_1, z_1, z_1]^{\alpha_2}[z_i, z_1, z_i, z_i]^{\alpha_3},
\end{equation}
\noindent
where
\begin{equation}
\label{eq:19}
\alpha_1 =  (k_1m_i-k_im_1)(k_1m_i+k_im_1) , \alpha_2=(k_1m_i-k_im_1)k_1m_1  , \alpha_3  =  (k_1m_i-k_im_1)k_im_i.
\end{equation} A solution of  (\ref{eq:19} corresponds to values: $\alpha_1=1, \alpha_2=\alpha_3=0.$    Note that $\delta = (k_1m_i-k_im_1) \not=0, \delta \in \{\pm1\},$ thus, $k_1m_1= k_im_i = 0.$ Consider two the cases. 

1) $k_1=0 \Rightarrow \delta = -k_im_1 \Rightarrow k_i\not= 0\,  \& \,  m_i=0 \Rightarrow \alpha_1 = -k_i^{2}m_1^2 = 1,$
contradiction.

2) $m_1= 0 \Rightarrow \delta = k_1m_i \not= 0  \Rightarrow m_i\not= 0\,  \&\,  k_i = 0 \Rightarrow \alpha_1 = k_1^{2}m_i^2 = 1 \Rightarrow k_1=\varepsilon_1\in \{\pm 1\}, m_i=\varepsilon_i\in \{\pm 1\},$ which gives  a solution of (\ref{eq:18}).

Consider  (\ref{eq:16}) again.
It follows from our proof that the image of  $\tilde{h}_1$ in $A_r$ does not depend of $a_i,$ and the image of   $\tilde{h}_i$ does not depend of $a_1.$ Obviously these both the images as image of any element of   $H$, does not depend of   $a_{l+2 }, ...,a_r.$ The image of  $\tilde{h}_j$ includes  $a_j$ in exponent  
$\varepsilon_j\in \{\pm 1\}, j =1, i.$

Consider a system of equations 
$$ [h_i(x_1, ..., x_r), h_1(x_1, ..., x_r), h_1(x_1, ..., x_r), h_i(x_1, ..., x_r)] = [h_i, h_1, h_1,
 h_i], $$
 \begin{equation}
 \label{eq:20}
 i = 2, ..., l+1.
 \end{equation}

 Here as above  $h_i(x_1, ..., x_r)$ corresponds to  $h_i=h_i(z_1, ..., z_r)$
 that is an expression of    $h_i$  by the basic elements of   $G$.

 Since $G$ is $l$-verbally closed this system has a solution in   $G$: $x_j=z_j, j = 1, ..., r.$ Hence it has a solution in  $H$: $x_j = f_j, f_j \in H, j = 1, ..., r.$ Denote $u_i = h_i(f_1, ..., f_r), i = 1, ..., l+1.$
Let $\nu $ be an endomorphism of  $G$ defined by the map $\nu : z_j \mapsto f_j, j = 1, ..., r.$ Note that $u_j=\nu (h_j(z_1, ..., z_r)=\nu (h_j), \ j = 1, ..., r.$
Then  $\nu (G)\leq H$. The following equalities are valid:
 \begin{equation}
 \label{eq:21}
  [u_i, u_1, u_1, u_i] = [h_i, h_1, h_1, h_i], i = 1, ..., l+1.
  \end{equation}
  
 In other words $\nu $ fixes every of elements  $[h_i, h_1, h_1, h_i], i = 1, ..., l+1.$
 The same is true for   $\eta =\nu^2$.
  By arguments similar to ones about  \cite{eq:15} applied to each of equations   (\ref{eq:20}) we get  congruence
\begin{equation}
\label{eq:22}
 u_1 \equiv h_1^{\varepsilon_1} (mod \ G'), \varepsilon_1\in \{\pm 1\}.
 \end{equation}
 By direct computation we obtain congruences $u_i \equiv h_i^{\varepsilon_i} (mod \ G'), \varepsilon_i\in \{\pm 1\}, i =2, ..., l+1.$ Let for example  $u_2\equiv h_2^{varepsilon_2}h_3^{\lambda_3} ... h_{l+1}^{\lambda_{l+1}, \, \varepsilon_2\in \{\pm 1}$ (this was proved above), $\lambda_i\in \mathbb{Z}$ and at least one  of $\lambda_j$ is not $0$, let be  $\lambda_3.$ Then in expression of $[u_2,u_1,u_1,u_2]$ as a product of exponents of basic commutators modulo $\gamma_5H$ of $[u_2,u_1,u_1,u_2]$ we have  $[h_3, h_1, h_1, h_3]$ with exponent 
 $\lambda_3^2\not= 0,$ that contradics to (\ref{eq:21}).

 Denote $w_j=\nu (u_j) =\eta (h_j),$ then  $ w_i \equiv z_i (mod \  G')$ and $w_i\equiv h_i (mod\ H'), i=1, ..., l+1.$
 We have a system of equalities
 \begin{equation}
 \label{eq:23}
  [w_i, w_1, w_1, w_i] = [h_i, h_1, h_1, h_i], i = 1, ..., l+1.
  \end{equation}

 We use congruences $w_i \equiv h_i (mod \ H'), \ i=1, ..., l+1.$  Let $w_i=c_ih_i, c_i\in H', i =1, ..., l+1.$ Then
\begin{equation}
\label{eq:18}
  [w_i, w_1, w_1, w_i] = [h_i, h_1, h_1, h_i]c_i^{(1-a_1)^2(1-a_i)}c_1^{(1-a_i)^2(a_1-1)},  i = 2, ...,
  l+1.
  \end{equation}

  After cancellations in exponents we obtain
  \begin{equation}
  \label{eq:19}
  c_1^{1-a_i}=c_i^{1-a_1}, i = 2, ..., l+1.
  \end{equation}
  Hence there is    $d\in H'$   such that $c_1=d^{1-a_1}, c_i=d^{1-a_i}, i=2, ..., l+1.$
 The $\eta $ is inner automorphism  corresponding to conjugation by  $d\in H'.$ Define endomorphism $\psi$ as composition of  $\eta$ with inner automorphism corresponding to conjugation by   $d^{-1}.$ Then $\psi$ maps  $G$ to $H$ and is identical on  $H_{l+1}.$

 It remains to prove that   $\psi$ is identical on 
   $H.$

Let $h=h_{l+k}\, (k=2, ..., t)$ be one of the chosen generators of   $H.$
 Then there is  $0\not= \alpha_k \in \mathbb{Z}[a_1^{\pm 1}, ..., a_{l+1}^{\pm 1}]$, such that
 $h^{\alpha_k}\in H_{l+1}.$

   Let $h(z_1,…,z_r)$ be an expression of  $h$ as a word in basic elements of 
  $G.$ Then
\begin{equation}
\label{eq:20}
h(x_1,…,x_r) = h
\end{equation}
\noindent
 is solvable in  $G.$ Hence it has a solution 
 $f_1,…,f_r$ in $H.$ We substitute this solution to  (\ref{eq:20}) and select in the left side  $h_{l+j}, \, j=2, ..., t.$ We obtain 
\begin{equation}
\label{eq:27}
 \tilde{h}\prod_{j=2}^th_{l+j}^{\alpha_{kj}} = h, \, \tilde{h}\in H_{l+1}, \alpha_{kj}\in \mathbb{Z}[A_{l+1}]=\mathbb{Z}[a_1^{\pm 1}, ..., a_{l+1}^{\pm 1}].
 \end{equation}
We write this for every equation in  (\ref{eq:27}) for $k = 1, ..., t,$ and obtain a system
\begin{equation}
\label{eq:28}
\tilde{h}_{l+k}=h_{l+k}^{1- \alpha_{kk}}\prod_{j=2, j\not= k}^th_{l+j}^{-\alpha_{kj}}, \, \tilde{h}_{l+k}\in H_l, \alpha_{kj}\in \mathbb{Z}[A_l]=\mathbb{Z}[a_1^{\pm 1}, ..., a_l^{\pm 1}], \, k=2, ..., t.
\end{equation}

 It remains to note that all  $\alpha_{ij}$  in this equality lie in the fundamental ideal  of the group ring
 $\mathbb Z[A_{l+1}],$ because
 $h_{l+k} (x_1,…,x_r)$ is commutator word for every   $k=2, ..., t.$ Then ${1-\alpha_{kk}}$ is
$1$ modulo the fundamentsl ideal. It follows that determinant  of the matrix consisting of exponent of the right side in (\ref{eq:28}) is not $0.$ If consider  $h_{l+k}, \, k =2, ..., t,$ as unknowns the system is solvable over the field of fractions of the ring $\mathbb{Z}[A_{l+1}].$ Applying  to (\ref{eq:28}) the Gauss elimination process we get
\begin{equation}
\label{eq:29}
h_{l+k}^{\alpha_k}=\tilde{h}_k, \, \alpha_k \in \mathbb{Z}[A_l], \, \tilde{h}_k\in H_{l+1}, \, k=2, ..., t.
\end{equation}

The endomorphism $\psi $ acts identically to $H_{l+1}. $  Hence it induces identical map to 
 $A_{l+1}.$ Then
\begin{equation}
\label{eq:30}
\eta (h_{l+k}^{\alpha_k}) = (\eta (h_{l+k}))^{\alpha_k} = h_{l+k}^{\alpha_k}, \, k = 2, ..., t.
\end{equation}
It is known that $G'$ has no module tersion. Since  $\psi (h_{l+k})=h_{l+k}$ for any
$k = 2, ..., t,$ then $\psi $ is identical on $H.$ This means  that $\psi $ is a retraction  $G\rightarrow H$.

$\square$

\bigskip

\begin{center}

{\bf \S3. PROBLEMS}

\end{center}

We put a few of natural problems.

\begin{problem} Is it true that any finitely generated verbally closed subgroup of a free solvable group $S_{r,d}$ is a retract?  What is the answer in a specific case $d=2$, i.e., in the case of a free metabelian group $M_r$?
\end{problem}
More specific question:
\begin{problem} Is it true that a subgroup  $H$ of a free solvable group $S_{r,d}$ of class $d \geq 3,$
such that  $r_{ab}(H) =2$ is verbally closed if and only if $H$ is a retract of 
 $S_{r,d}?$
\end{problem}
A positive answer will strengthen  Theorem \ref{th:3}.


\begin{thebibliography}{9}
\bibitem{Hod} {\it W. Hodges}, Model Theory, Cambridge, Cambridge Univ. Press, 1985.


\bibitem{Scott} {\it W.R. Scott}, Algebraically closed groups, Proc. Amer. Math. Soc., {\bf 2}\ (1951), 118--121.

\bibitem{MR}  {\it A. Myasnikov, V. Roman'kov}, Verbally closed subgroups of free groups, J. Group Theory, {\bf 17}, No. 1\ (2014), 29--40.

\bibitem{Berg} {\it  G. M.  Bergman}, Supports of derivations,  free factorizations and ranks of fixed subgroups in free groups, Trans. Amer. Math. Soc., {\bf 351}\  (1999), 1551--1573.

\bibitem{KT} The Kourovka Notebook. Unsolved problems in group theory. No. 19 (2018),Editors: E.I. Khukhro and V.D. Mazurov.  Sobolev Institute of Math. Siberian Branch RAS, Novosibirsk.

\bibitem{Open} {\it G. Baumslag,  A. Myasnikov,  V. Shpilrain,} Open problems in combinatorial group theory, Contemporary Math., {\bf 296}\  (2002),  1--38.

\bibitem{Retract} {\it I. Snopce, S. Tanushevski, P. Zalesskii},  Retracts of free groups and a question of Bergman, arXiv:1902.02378 [math.GR]. Feb 6 (2019).

\bibitem{RomKhis} {\it V.A. Roman'kov, N.G. Khisamiev,} Verbally and existentially closed subgroups of free nilpotent groups.
Algebra and Logic. 2013. {\bf 52}, No. 4, 336--351. 
.

\bibitem{Maz1} A.  {\it M. Mazhuga}, On free decompositions of verbally closed subgroups in free products of finite groups, J. Group Theory,  {\bf 20}, No. 5\ (2017), 971--986.

\bibitem{Maz2} {\it A.M. Mazhuga}, Strongly verbally closed groups, J. Algebra, {\bf 493}\  (2018), 171--184.

\bibitem{Mthesis} {\it A.M. Mazhuga}, Verbally closed subgroups,Thesis, Lomonosov Moscow State University, Moscow, 2018 (In Russian).  

\bibitem{KlMaz} {\it Ant. A. Klyachko, A.M. Mazhuga}, Verbally closed virtually free subgroups, Sbornik: Mathematics, {\bf 209}, No.  6\  (2018),   850--856.

\bibitem{KMM} {\it A. A.Klyachko, A. M. Mazhuga, V. Y. Miroshnichenko},
Virtually free finite-normal-subgroup-free groups are strongly verbally closed, J. Algebra, {\bf 510}\ (2018), 319--330.

\bibitem{Timmet} {\it E.I. Timoshenko}, Test elements and test rank of a free metabelian group, Siberian Mathematical Journal,  {\bf 47}, No. 6\  (2000), 1200--1204.



\bibitem{Romtest} {\it V.A. Roman'kov,}  Test elements for free solvable groups of rank 2.
Algebra and Logic. 2001. {\bf 40}, No.2.  106--111. 

\bibitem{Timsolv} {\it E.I. Timoshenko}, Computing test rank for a free solvable group. Algebra and Logic,  {\bf 45},No. 4\  (2006), 254--260.


\bibitem{Timpol} {\it C.K. Gupta, E.I. Timoshenko,} Test rank for fome free polynilpotent Groups, Algebra and Logic{\bf 42}, No. 1\ (2003),  20--28.


\bibitem{Timbook} {\it  E.I. Timoshenko},  Endomorphisms and universal theories of solvable groups. Novosibirsk State Technical University,  Novosibirsk, 2011 (In Russian).

\bibitem{RemSoc} {\it V.N. Remeslennikov, V.A. Sokolov}, Some properties of Magnus embedding,
Algebra and Logic,   {\bf 9} (1970), 342--349.

\bibitem{KrFox}{\it  R.H. Crowell, R.H. Fox}, Introduction to knot theory, Springer-Verlag, Berlin-Heidelberg- New York,  1963.

\bibitem{Romess} {\it V. A. Roman'kov},  Essays in algebra and cryptology. Solvable
groups, Omsk, Omsk State University Publishing House, 2017.



\bibitem{B} {\it G. Baumslag},  Some subgroup theorems for free ${\upsilon}$-groups, Trans. Amer. Math. Soc., {\bf 108}\  (1963), 516--525.

\bibitem{Romeqmet} {\it V.A. Roman'kov}, Equations in free metabelian groups,
Siberian Mathematical Journal. 1979. {\bf  20}, No. 3 (1979).  469--471. 

\bibitem{HN} {\it H. Neumann}, Varieties of groups, Springer-Verlag, Berlin-Heidelberg,  1967.

\bibitem{Art} {\it V.A. Artamonov}, 
Projective metabelian groups and Lie algebras, Uspekhi Mat. Nauk, {\bf 32}, No. 3(195) (1977), 166.

\bibitem{L} {\it J. Lewin}, A note on zero divisors in group rings, Proc. Amer. Math. Soc., {\bf 31} \ (1972), 357--359.

\bibitem{Romanovski} {\it C. K. Gupta, N.S. Romanovski}, On torsion
in factors of polynilpotent series of a group with a single
relation, International Journal of Algebra and Computation, {\bf
14}, No. 4\ (2004), 513-523.
\end{thebibliography}
\end{document}